\documentclass[preprint,12pt]{elsarticle}

\usepackage{amssymb}
\usepackage{amsmath}
\usepackage{amsfonts}
\usepackage{graphicx}
\usepackage{graphics}
\usepackage{tikz}
\usepackage{titlesec}
\usepackage[english]{babel}

\newtheorem{theorem}{Theorem}

\newtheorem{lemma}[theorem]{Lemma}

\newtheorem{remark}[theorem]{Remark}

\def\qed{\vbox{\hrule
 \hbox{\vrule\hbox to 5pt{\vbox to 8pt{\vfil}\hfil}\vrule}\hrule}}

\usepackage{tikz}
\usepackage{tkz-graph}

\journal{Linear Algebra and Its Applications}

\begin{document}
\begin{frontmatter}

\title{An increasing sequence of lower bounds for the Estrada index of graphs and matrices}

\author[]{Juan R. Carmona\corref{cor1}}
\ead{juan.carmona@uach.cl}
\address{Instituto de Ciencias F\'isicas y Matem\'aticas,\\
  Universidad Austral de Chile, Independencia 631 - Valdivia - Chile}
\cortext[cor1]{Corresponding author}

\author{Jonnathan Rodr\'{\i}guez}
\ead{jrodriguez01@ucn.cl}
\address{Departamento de Matem\'{a}ticas, Universidad Cat\'{o}lica del Norte Av. Angamos 0610 Antofagasta, Chile, \\Departamento de Matem\'{a}ticas, Universidad de Antofagasta, Av Angamos 601, Antofagasta, Chile.}

\begin{abstract}
Let $G$ be a graph on $n$ vertices and $\lambda_1\geq \lambda_2\geq \ldots \geq \lambda_n$ its  eigenvalues. The Estrada index of
$G$ is defined as $EE(G)=\sum_{i=1}^n e^{\lambda_i}.$ In this work, we use an increasing sequence converging to the $\lambda_1$ to obtain an increasing sequence of lower bounds for $EE(G)$. In addition, we generalize this succession for the Estrada index of an arbitrary symmetric  nonnegative  matrix.
\end{abstract}


\end{frontmatter}
\section{Introduction}

In this paper, we consider undirected simple graphs $G$ with edge set denoted by $\mathcal{E}(G)$ and its vertex set $V(G)=\{ 1, \ldots, n \}$ with cardinality $m$ and $n$, respectively. 
	If $ e \in \mathcal{E}(G)$ has end vertices $i $ and $j$, then we say that $i$ and $j$ are adjacent and this edge is denoted by $ij$. For a finite set $U$, $|U|$ denotes its cardinality.
	Let $K_n$ be the complete graph with $n$ vertices and $\overline{K_n}$ its complement. A graph $G$ is bipartite if there exists a partitioning of $V(G)$ into disjoint, nonempty sets $V_1$ and $V_2$ such that the end vertices of each edge in $G$ are in distinct sets $V_1$, $V_2$. A graph $G$ is a complete bipartite graph if $G$ is bipartite and each vertex in $V_1$ is connected to all the vertices in $V_2$. If $|V_1|=p$ and $|V_2|=q$, the complete bipartite graph is denoted by $K_{p,q}$. For more properties of bipartite graphs, see \cite{haemers}.
	
	If $i \in V(G)$, then $NG(i)$ denoted the set of neighbors of the vertex $i$ in $G$, that is, $NG(i) = \{j \in V(G): ij \in \mathcal{E}(G)\}$.
	For the $i$-th vertex of $G$, the cardinality of $NG(i)$ is called the degree of $i$ and it is denoted by $d(i)$.\\
	The number of walks of length $k$ of G starting at $i$ is denoted by $d_k(i)$, is also called $k$-degree of the vertex $i$. Clearly, we define $d_0(i) = 1 ,d_1(i) = d(i)$ and {for $k \geq 1$} $$d_{k+1}(i)= \sum_{j \in NG(i)}d_k(j).$$
	
	 A graph $G$ is called $r$-regular if 
    $d(i) = r,$ for all $i \in V(G)$. Further, a graph $G$ is called $(a, b)$-semiregular
    if $\{d(i), d(j)\} = \{a, b\}$ holds for all edges $ij \in \mathcal{E}(G)$. A semiregular graph that is not regular will henceforth be called strictly semiregular. Clearly, a connected strictly semiregular graph must be bipartite. A graph $G$ is called harmonic \cite{D-G} (in \cite{Y-L-T} is called pseudoregular) if there exist a constant $\mu$ such that $d_2(i) = \mu d(i)$ holds for all $i \in V (G);$ in which case $G$ is also called $\mu$-pseudoregular.
A graph G is called semipseudoregular \cite{Y-L-T}, if there exist a constant $\mu$ such that $d_3(i) =\mu d(i)$ holds for all $i \in V (G);$ in which case $G$ is also called $\mu$-semiharmonic.
Thus every $\mu$-pseudoregular graph is $\mu^2$-semipseudoregular. Also every
$(a, b)$-semiregular graph is $ab$-semipseudoregular. A semipseudoregular graph that is not pseudoregular
will henceforth be called strictly semipseudoregular.

Finally, a graph $G$ is called $(a, b)$-pseudosemiregular if $\displaystyle \left\{ \frac{d_2(i)}{d(i)}, \frac{d_2(j)}{d(j)}\right\}=\{a,b\}$ holds for all edges $ij \in \mathcal{E}(G)$. A pseudosemiregular graph that is not pseudoregular will henceforth be called strictly pseudosemiregular. Clearly, a connected strictly pseudosemiregular graph must be bipartite.\\

The  adjacency  matrix $A(G)$ of  the graph $G$ is  a 
symmetric matrix of order $n$ with entries $a_{ij}$, such that $a_{ij}=1$ if $ij \in \mathcal{E}(G)$ and $a_{ij}=0$ otherwise. Denoted by $\lambda_1\geq \ldots \geq \lambda_n$ to the eigenvalues of $A(G)$,  
see \cite{C-D-S1, C-D-S2}.
  
The Estrada index of the graph $G$ is defined as
	$$EE(G)= \sum_{i=1}^n e^{\lambda_i}.$$
	This spectral quantity is put forward by E. Estrada \cite{E1} in the year 2000. There have been found a lot of chemical and physical applications, including quantifying the degree of folding of long-chain proteins, \cite{E1, E2, E3, G-F-G-M-V, G-R-F-M-S, G-G-M-S}, 
	complex networks and evolving graphs \cite{E4, E5, S1, S2, S3, S4, Shang}. 
	Mathematical properties of this invariant can be found in e.g. \cite{F-A-G, G-D-R, G-R, K, S5, Z, Z-Z}.

Denote by $M_k = M_k(G)$ to the $k$-th spectral moment of the graph $G$, i.e.,
$$M_k= \displaystyle
4
 \sum_{i=1}^n(\lambda_i)^k.$$ In \cite{C-D-S1}, for a graph $G$ with $n$ vertices and $m$ edges, the authors proved that
\begin{equation}\label{eq1}
M_0 =n,\,\, M_1 =0,\,\, M_2 =2m,\,\, M_3 =6t,
\end{equation}
where $t$ is the number of triangles in $G.$ Thereby we can write the Estrada index as
$$EE(G) = \sum_{k=0}^{\infty}\frac{M_k}{k!}.$$

\section{Preliminaries}

Let $R$ be a symmetric matrix of order $\ell$ with eigenvalues $\rho_1, \rho_2,\ldots , \rho_{\ell}$. 
The Estrada index of $R$ is denoted and defined as 
\begin{equation}\label{eer}
    EE(R)=\sum_{i=1}^{\ell}e^{\rho_i}.
\end{equation}

For more details on the theory of the Estrada index see the papers \cite{J-J-J,D-S-E-K-F-N,R} and the references cited therein.
\\

Among the results obtained for the Estrada index of symmetric matrices, we outline the following lower bound obtained in \cite{J-J-J}
\begin{eqnarray}
\label{JJJ2}
	EE(R)\geq e^{\rho_1}+(\ell-1)+Tr(R)-\rho_1.
\end{eqnarray}
For the other hand, considering $R$ as the adjacency matrix of a graph $G$, in the same paper, the authors obtain the following lower bound for the Estrada index of a graph $G$

\begin{eqnarray}
\label{JJJ}
	EE(G)\geq e^{\left( \frac{2m}{n}\right)}+n-1-\frac{2m}{n}.
\end{eqnarray}\
The equality holds in (\ref{JJJ}) if and only if $G$ is isomorphic  to $\overline{K_n}$.
To obtain the last lower bound, at first, the authors showed that the following relationship holds:
\begin{equation}\label{1}
	EE(G)\geq e^{\lambda_1}+n-1-\lambda_1.
\end{equation}

In \cite{C-D-S1}, the authors obtain a sequence of lower bounds for Estrada index of graphs. For this, we need the following results.
\begin{theorem}{\cite{C-D-S1}}
Let $G$ be a connected graph. Then
\begin{equation}\label{3}
\lambda_1 \geq \sqrt{\displaystyle \frac{\displaystyle \sum_{i \in V(G)} d^2(i)}{n}}
\end{equation}
with equality if and only if $G$ is regular or a semiregular.

\end{theorem}
\begin{theorem}{\cite{Y-L-T}}
Let $G$ be a connected graph. Then
\begin{equation}\label{3}
\lambda_1 \geq \sqrt{ \frac{\displaystyle \sum_{i \in V(G)} d_2^2(i)}{\displaystyle \sum_{i \in V(G)} d^2(i)}}
\end{equation}
with equality if and only if $G$ is a pseudo-regular graph or a strictly pseudo-semiregular graph. 

\end{theorem}
Furthermore, we have to $$\sqrt{ \frac{\displaystyle \sum_{i \in V(G)} d_2^2(i)}{\displaystyle \sum_{i \in V(G)} d^2(i)}}\geq \sqrt{ \frac{\displaystyle \sum_{i \in V(G)} d^2(i)}{n}}.$$
In \cite{H-T-W}, the authors generalized these results and built an increasing sequence of lower bounds for $\lambda_1$, as follows: 
\begin{equation}\label{seq1}
  \begin{array}{cc}
        \gamma^{(0)}=\sqrt{ \displaystyle\frac{ \displaystyle \sum_{i \in V(G)} d^2(i)}{n}} \\
            \\
        \gamma^{(1)}=\sqrt{ \frac{\displaystyle \sum_{i \in V(G)} d_2^2(i)}{\displaystyle \sum_{i \in V(G)} d^2(i)}} \\
        \vdots \\
        \gamma^{(k)}=\sqrt{ \frac{\displaystyle \sum_{i \in V(G)} d_{k+1}^2(i)}{\displaystyle \sum_{i \in V(G)} d^2_{k}(i)}}.
  \end{array}  
\end{equation}

Thereby, they obtain the following results.
\begin{theorem}{\cite{H-T-W}}\label{Lem}
Let $G$ be a connected graph and $k \geq 1$. Then
$$ \lambda_1 \geq \gamma^{(k)} $$
with equality if and only if $G$ is pseudoregular or strictly pseudosemiregular.
\end{theorem}

\begin{theorem}{\cite{H-T-W}}\label{T1}
Let G be a graph, then $\{\gamma^{(k)}\}_{k\geq 0}$ is an increasing sequence and $$\displaystyle \lim_{k\rightarrow \infty}\gamma^{(k)}=\lambda_1.$$
\end{theorem}

In this work, following the ideas of \cite{C-D-S1,C-G-T-R}, we will use the increasing sequence of lower bounds for $\lambda_1$ given in \cite{H-T-W} to obtain an increasing sequence of lower bounds for the Estrada index of graphs, which converges to the lower bound (\ref{1}). Moreover, applying this technique, we obtain an increasing sequence of lower bounds for the Estrada index of a symmetric matrix $R$, which converges to the lower bound (\ref{JJJ2}).

\section{Main Results}

\noindent Consider the following function 
\begin{equation}\label{f}
f(x)= (x-1)-\ln(x), \, \, x>0.
\end{equation}
Clearly the function $f$ is decreasing in $(0,1]$ and increasing in $[1,+\infty)$
Consequently, $f(x) \geq f(1) = 0$, implying that 
\begin{equation}\label{eq7}
x \geq 1 + \ln x, \,\,\,x > 0,
\end{equation}
the equality holds if and only if $x = 1$. Let $G$ be a graph of order $n$, using (\ref{eq1}) and (\ref{eq7}), we get:
\begin{equation}\label{Energia0}
\begin{array}{lllll}
EE(G) & \geq & e^{\lambda_1} + (n-1)+ \displaystyle \sum_{k=2}^n \ln e^{\lambda_k}  \\
& = & e^{\lambda_1} + (n -1) +\displaystyle \sum_{k=2}^n \lambda_k \\
& = &  e^{\lambda_1} + (n -1) +M_1 - \lambda_1\\
		
& = &  e^{\lambda_1} + (n -1)  - \lambda_1.
\end{array}
\end{equation}
Define the function
\begin{equation}\label{fi}
\phi(x)= e^x + (n-1) - x,\qquad x>0.
\end{equation}
		
\noindent Note that, this is an increasing function on $D_{\phi}=[0,+\infty)$.

Therefore, we proved the following result.

\begin{theorem}\label{Tg}
Let $G$ a connected graph of order $n$ and the sequence $\{\gamma^{(k)}\}_{k=0}^{\infty}$ as in (\ref{seq1}). Then the sequence $\{\phi(\gamma^{(k)})\}_{k=0}^{\infty}$ is increasing and converges to $\phi (\lambda_1),$ 
moreover for all $ k\geq 0$ 
		\begin{equation}\label{cota0}
		EE(G)> e^{\left( \gamma^{(k)}\right)}+n-1-\gamma^{(k)}.
		\end{equation}
\end{theorem}		

	\begin{proof}
	First, we observe that $\gamma^{(k)} \in D_{\phi}$, for all $ k\geq 0.$ 
This is an immediate consequence of Theorem \ref{T1} and that $\gamma^{(0)}=\displaystyle \sqrt{\frac{\sum_{i \in V(G)} d^2(i)}{n}}  \geq \sqrt{\frac{2m}{n}} \geq 1.$ Since  $\{\gamma^{(k)}\}_{k=0}^{\infty}$ is an increasing sequence and by Theorem \ref{T1} converges to $\lambda_1.$ Then, by the continuity of $\phi,$ we have to 
$$\displaystyle \lim_{k \rightarrow \infty} \phi(\gamma^{(k)})=\phi\left(\lim_{k \rightarrow \infty}\gamma^{(k)}\right)=\phi (\lambda_1).$$
This allows us to prove the first statement.
\end{proof}

\vspace{0.2cm}

\begin{remark}
Suppose that the equality holds in (\ref{cota0}). Then all the inequalities in (\ref{Energia0}) must be considered as equalities. From the equality (\ref{eq7}), we get $e^{\lambda_2}=\ldots=e^{\lambda_n}=1,$ then $\lambda_2=\ldots={\lambda_n}=0$ implying that $\lambda_1=\gamma^{(k)}=0.$ Thus $G$ is isomorphic to the $\overline{K_n}.$
\end{remark}

\noindent The  following  result, a sharp increasing sequence of lower bounds for the Estrada index of a bipartite graph is obtained.\\

Considering (\ref{eq1}) and (\ref{eq7}), we obtain
	\begin{equation}\label{EE2}
	\begin{array}{lllll}
	EE(G) & = & e^{\lambda_1}+e^{-\lambda_1} + \displaystyle \sum_{k=2}^{n-1} e^{\lambda_k}  \\
	& \geq & 2\cosh{\lambda_1} + (n-2) + \displaystyle \sum_{k=2}^{n-1} \lambda_k \\
	& = &  2\cosh{\lambda_1} + (n-2) + M_1 + \lambda_1 - \lambda_1\\
	& = &   2\cosh{\lambda_1} +n-2.
	\end{array}
	\end{equation}	
	Define the function, $\Phi(x)= 2\cosh{x} + n-2.$
	%
	\noindent Note that, this is an increasing function on $D_{\Phi}=[0,+\infty)$.
\begin{theorem}\label{Tgb}
	Let $G$ be a bipartite connected graph of order $n$ with $n>2$. Considering
the sequence $\{\gamma^{(k)}\}_{k=0}^{\infty}$ as in (\ref{seq1}). Then  
			\begin{equation}\label{cota1}
		EE(G)\geq 2\cosh{\left(\gamma^{(k)}\right)}+n-2.		
		\end{equation}
		Equality holds 
		 if and only if G is isomorphic to the complete bipartite graph.
\end{theorem}		
	\begin{proof}

	Analogous to the proof of Theorem \ref{Tg}, $\gamma^{(k)} \in D_{\Phi}$. Since  $\{\gamma^{(k)}\}_{k=0}^{\infty}$ is an increasing sequence and by Theorem \ref{T1} converges to $\lambda_1.$ Then 
the continuity of $\phi$ allows us to prove the first statement.\\
		
		Moreover, we get
		$$EE(G)\geq 2\cosh{\left(\gamma^{(k)}\right)}+n-2.$$	
			Suppose now that the equality holds in (\ref{cota1}). Then all the inequalities in (\ref{EE2}) must be considered as equalities. From the equality (\ref{eq7}), we get $e^{\lambda_2}=\ldots=e^{\lambda_{n-1}}=1,$ then $\lambda_2=\ldots={\lambda_{n-1}}=0$ and $\lambda_1=-\lambda_n=\gamma^{(k)}$, which implies that $G$ is a bipartite complete graph, $K_{p,q}$ such that $p+q=n$ 
	\end{proof}
\newpage
For the other hand, recall  that  a  Hermitian complex  $\ell \times \ell$ matrix $R=  (r_{ij})$,  $r_{ij} \in \mathbb{C}$,  is  such  that $R=R^*$ where $R^{*}$ denotes  the  conjugate  transpose  of $R$.  For $x,  y \in \mathbb{C}^{\ell}$,  we  denote  by $<x,y>=x^{*}y$, the  inner  product  in $\mathbb{C}^{\ell},$ the norm, $\|x\|=\sqrt{<x,  x>}$ and $|R|=\sqrt{trace (R^*R)}$ the Frobenius norm  of $R$.

For  symmetric  matrices  it  is  possible  find  an  orthonormal  basis  of $C^{\ell}$, $x_1,  x_2,\ldots,x_{\ell}$,  of eigenvectors associated to  $\rho_1,  \rho_2,\ldots,\rho_{\ell}$ the  eigenvalues of $R$ 
where $Rx_i=\rho_ix_i$,  for $i=  1,2,\ldots\ell.$ The spectral radius of a square matrix $R$ is the largest absolute value of its eigenvalues, it is denoted by $|\rho_1|=  \max\{|\rho|:\rho_i \in \sigma(R)\}$. 
Let $f \in \mathbb{C}^{\ell}$ such  that $<f,x_i>\neq  0$  for $i=1,\ldots,\ell$.  We define the vector sequence 
\begin{equation}\label{seq2}
  \begin{array}{lllll}
        r^{(0)}=f \\
            \\
        r^{(1)}=Rr^{(0)}= Rf  \\
        \vdots \\
        r^{(k)}=Rr^{(k-1)}=R^{(k)}f
  \end{array}  
\end{equation}

In  view  of  the  fact  that  for $k=  0,1,\ldots,R^{(k)}f \neq  0$,  in \cite{C-G-T-R} the following numerical sequence was defined
\begin{equation}\label{seq3}
  \begin{array}{lllll}
        \xi^{(0)}=\displaystyle \frac{\|r^{(1)}\|}{\|r^{(0)}\|}=\frac{\|Rr^{(0)}\|}{\|f\|} \\
            \\
          \xi^{(1)}=\displaystyle \frac{\|r^{(2)}\|}{\|r^{(1)}\|}=\frac{\|Rr^{(1)}\|}{\|Rr^{(0)}\|}  \\
        \vdots \\
         \xi^{(k)}=\displaystyle \frac{\|r^{(k+1)}\|}{\|r^{(k}\|}=\frac{\|Rr^{(k)}\|}{\|Rr^{(k-1)}\|}. 
  \end{array}  
\end{equation}
Thereby, the authors in \cite{C-G-T-R},  we obtain the following result.
\begin{lemma}{\cite{C-G-T-R}}\label{crec}
Let $\{\xi^{(k)}\}_{k\geq 0}$ be the sequence defined in (\ref{seq3}), then $\{\xi^{(k)}\}_{k\geq 0}$ is an increasing sequence. Furthermore $$\displaystyle \lim_{k\rightarrow \infty}\xi^{(k)}=|\rho_1|.$$
\end{lemma}

If the matrix $R$ is an irreducible nonnegative  matrix, then its spectral radius is simple with a positive eigenvector $x_1,$ see \cite{M}.  
Let $\textbf{e}$ be the $n$-dimensional all-one  vector $\textbf{e}=  (1,\ldots,1)^T$.  Evidently,  $<\textbf{e}, x_1>\neq  0$.

Considering the argument used in (\ref{Energia0}) and the inequality (\ref{eq7}), we have that
\begin{equation}\label{lbrho}
    EE(R) \geq e^{\rho_1} + \ell -1 +Tr(R)-\rho_1.
\end{equation} 
The equality holds if only if $\rho_2=\rho_3=\ldots=\rho_{\ell}=  0.$ 

Let 
\begin{equation}\label{varfi}
    \varphi(x)= e^{x} + \ell -1 +Tr(R)-x.
\end{equation}
an increasing function on $D_{\varphi}=[0,+\infty)$. Then, the following result is obtained.

\begin{theorem}\label{T3}
Let $R$ be a nonnegative  symmetric $\ell \times \ell $ 
matrix  with  spectral  radius $\rho_1$. Define the  sequence $\{\xi^{(k)}\}_{k\geq 0}$ as  in (\ref{seq3}) with $f$ replaced  by $\textbf{e}$.  Then $\xi^{(k)} \in D_{\varphi}$,  for  all $k=  0,1,2,...\ell,$ where $\varphi$ is  defined in (\ref{varfi}).
\end{theorem}
\begin{proof}
Before starting  the  proof,  it  is  worth  noting  that $$\rho_1 \leq \sqrt{\rho_1^2+\rho_2^2+...+\rho_{\ell}^2}=|R|.$$%
For  the  Rayleigh  quotient, we have that $0< \xi^{(k)} \leq |\rho_1|$ for all $k\geq 0$, then $0< \xi^{(k)}\leq |R|$. Since $R$ is  a symmetric nonnegative  matrix, $$trace(R^*R)=trace(R^2)\leq \textbf{e}^*R^2\textbf{e},$$ implying  that $|R| \leq \|R\textbf{e}\|$. Consequently,
\begin{equation}\label{e2}
    \displaystyle 0 \leq \frac{|R|}{\sqrt{\ell}} \leq \frac{\|R\textbf{e}\|}{\sqrt{\ell}}=\xi^{(0)}.
\end{equation}
where $\sqrt(\ell)$ is the norm of vector \textbf{e}, which is considered in (\ref{seq3}) for obtain the sequence of vectors $\xi^{(i)}$.
Finally,  by Lemma \ref{crec}, the  result  follows.
\end{proof}

The  following result is a generalize the Theorem \ref{Tg}.

\begin{theorem}
Let $R$ be  a  nonnegative  symmetric $\ell \times \ell $ 
matrix  with  spectral  radius $\rho_1$. Define the  sequence $\{\xi^{(k)}\}_{k\geq 0}$ as  in (\ref{seq3}) with $f$ replaced  by $\textbf{e}$. Then the sequence $\{\varphi(\xi^{(k)})\}_{k=0}^{\infty}$ is increasing and converges to $\varphi (\rho_1),$ 
moreover for all $ k\geq 0$  
\begin{equation}\label{ern}
    EE(R) \geq e^{\xi^{(k)}} + \ell -1 +Tr(R)- \xi^{(k)}.
\end{equation}
Equality holds in (\ref{ern}) if  and  only  if $\rho_2=\rho_3=\ldots=\rho_{\ell}=  0.$
\end{theorem}
\begin{proof}
Since $\xi^{(k)} \in D_{\varphi}$, for all $ k\geq 0,$ see Theorem \ref{T3}.
Since  $\{\xi^{(k)}\}_{k=0}^{\infty}$ is an increasing sequence and converges to $\rho_1,$ by Lemma \ref{crec}. Then 
the continuity of $\varphi$ allows us to prove the statement.
\end{proof}

\section{Comparing bounds and final remarks}

In this section, we show some examples of the results obtained in the previous sections for several graphs, see Figure \ref{fig1}. We compare the estimates obtained in Theorem \ref{Tg} and Theorem \ref{Tgb} with the exact value of the Estrada Indice. \\

Consider the complete graph $K_{4},$ with eigenvalues $\{3,-1^{\left[ 3\right] }\}.$ Then $$EE(K_{4})=e^{3}+3e^{-1}=%
\allowbreak 21.\,\allowbreak 189.$$ For the other hand, we have:

\[
\begin{tabular}{lrr}
$d(1)=$ $d(2)=d(3)=d(4)=3$ \\ 
$d_{2}(1)=$ $d_{2}(2)=d_{2}(3)=d_{2}(4)=3^{2};$ \\ 
$d_{3}(1)=$ $d_{3}(2)=d_{3}(3)=d_{3}(4)=3^{3};$ \\ 
$d_{4}(1)=d_{4}(2)=d_{4}(3)=d_{4}(4)=3^{4}$%
\end{tabular}%
\]

So, we can observe that
\[
d_{k}(1)=d_{k}(2)=d_{k}(3)=d_{k}(4)=3^{k} 
\]%
Thereby, we obtain: 
\[
\gamma ^{(0)}=\gamma ^{(1)}=\gamma ^{(k)}=\gamma ^{(\infty )}=\lambda _{1}=3 
\]

Therefore, applying the Theorem \ref{Tg}
\[EE(G) > J^{k}=e^{\left( \gamma ^{(k)}\right) }+n-1-\gamma ^{(k)},\] we have
\[
J^{0}=J^{1}=\cdots =J^{k}=J^{\infty }\approx 20.\,\allowbreak
086<EE(K_{4})= 21.\,\allowbreak 189 
\]

\begin{figure}[h]\label{fig1}
\begin{eqnarray*}
\begin{tikzpicture}
    \tikzstyle{every node}=[draw,circle,fill=black,minimum size=4pt,
                            inner sep=0pt]
    \draw

        (-4,0) node (1) [label=below:] {}
        (-2,0) node (2) [label=below:] {}
        (-2,2) node (3) [label=below:] {}
        (-4,2) node (4) [label=below:] {}
        (0,0) node (5) [label=below:] {}
        (0,1) node (6) [label=left:] {}
        (0,2) node (7) [label=left:] {}
        (0,3) node (8) [label=left:] {}
        (2,0) node (9) [label=below:] {}
        (2,1) node (10) [label=below:] {}
        (2,2) node (11) [label=below:] {}
        (4,1.5) node (12) [label=below:] {}
        (4,0.5) node (13) [label=below:] {};

        \draw (1)--(2); \draw (2)--(3);\draw (3)--(4);\draw (4)--(1);\draw (3)--(1);\draw (4)--(2);\draw (5)--(6);\draw (6)--(7);\draw (7)--(8);\draw (9)--(13);\draw (9)--(12);\draw (10)--(13);\draw (10)--(12);\draw (11)--(13);\draw (11)--(12);
            \end{tikzpicture}
\end{eqnarray*}
\caption{Graphs  $K_{4},P_{4}$  and   $K_{2,3}.$}%
\end{figure}
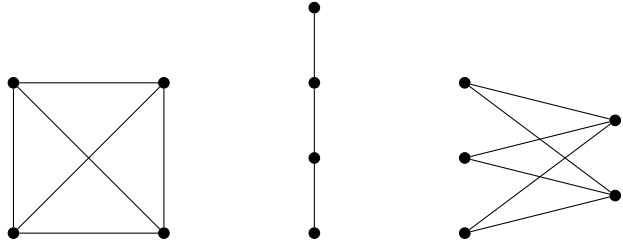

Consider the graph $P_{4},$ with eigenvalues $\{1.62,0.62,-0.62,-1.62\}.$ Then $$EE(P_{4})=2\cosh
(1.62)+2\cosh (0.62)=\allowbreak 7.\,\allowbreak 647\,9.$$ Note that, for $P_{4}$ we have

\[
\begin{tabular}{l}
$d(1)=1=d(4);$ $d(2)=2=d(3);$ \\ 
$d_{2}(1)=2=$ $d_{2}(4);$ $d_{2}(2)=3=d_{2}(3)$.%
\end{tabular}%
\]%
So, we can observe that 
\[
d_{k}(1)=d_{k}(4)=d_{k-1}(2)=d_{k-1}(3)\text{ \ }y\text{ \ }%
d_{k}(2)=d_{k}(3)=d_{k-1}(1)+d_{k-1}(3). 
\]%
Thereby, we obtain
\[
\gamma ^{(0)}=\sqrt{\frac{10}{4}}\approx 1.5811<\gamma ^{(1)}=\sqrt{\frac{26%
}{10}}\approx 1.6125<\cdots <\gamma ^{(k)}<\cdots <\gamma ^{(\infty
)}=\lambda _{1}\approx 1.62 
\]

Therefore, applying the Theorem \ref{Tgb}
\[ EE(G) > C^{k}= 2\cosh{\left(\gamma^{(k)}\right)}+n-2\] we have
\[
C^{0}\approx 6.\,\allowbreak 279\,2<C^{1}\approx 6.\,\allowbreak
402\,8<\cdots <C^{\infty }\approx 6.\,\allowbreak 433\,1<EE(P_{4}) =
\allowbreak 7.\,\allowbreak 647\,9 
\]

\bigskip 

Now, consider the complete bipartite graph $K_{2,3}$ with eigenvalues $\{\sqrt{6},0^{\left[ 3\right] },-\sqrt{6}\}.$ Then $$%
EE(K_{2,3})=2\cosh (\sqrt{6})+3=\allowbreak 14.\,\allowbreak 669\allowbreak .$$ Note that, for $K_{2,3}$ we have
\[
\begin{tabular}{lll}
$d(1)=d(2)=3$ ;$d(3)=d(4)=d(5)=2$ \\ 
$d_{2}(1)=$ $d_{2}(2)$ $=3\cdot 2;$ $\ d_{2}(3)=d_{2}(4)=d_{2}(4)=2\cdot 3$
\\ 
$\ \ \ \ \ \ \ \ \ \ \ \ \ \ \vdots $ \\ 
$d_{2k}(1)=d_{2k}(2)=2^{k}3^{k},$ $\ d_{2k}(3)=d_{2k}(4)=d_{2k}(5)=2^{k}3^{k}
$ \\ 
$d_{2k+1}(1)=d_{2k+1}(2)=2^{k}3^{k+1}$ , $%
d_{2k+1}(3)=d_{2k+1}(4)=d_{2k+1}(5)=2^{k+1}3^{k}$%
\end{tabular}%
\]

So, we can observe that

\[
\gamma ^{(0)}=\sqrt{\frac{30}{5}}=\sqrt{6}=\gamma ^{(1)}=\sqrt{\frac{180}{30}%
}=\sqrt{6}=...=\gamma ^{(k)}=\sqrt{6}=\lambda _{1}
\]

Therefore, applying the Theorem \ref{Tgb}
\[ EE(G) > C^{k}= 2\cosh{\left(\gamma^{(k)}\right)}+n-2\] we obtain

\[
C^{0}=C^{1}=C^{k}=C^{\infty }\approx 14.\,\allowbreak
669=EE(K_{2,3})= 14.\,\allowbreak 669
\]

Analyzing the above examples, we observe the following:
\begin{itemize}
\item The results obtained for symmetric matrices, in particular Theorems \ref{Tg} and Theorem \ref{Tgb} for graphs, the examples show that the increasing succession of lower bounds obtained for the Estrada index approximate in an optimal way and in the pertinent cases equality is reached.
\item
For future work, considering the technique exhibited demonstration, obtain a succession of upper bounds that delimit index of Estrada.
\item
Extend the results to other matrices associated with a graph, such as: Laplacian and signless Laplacian Matrix, among others.
\item
Study the Estrada index in relation to the evolving graph. Consider other possible combinatorial relationships to limit the Estrada index by other spectral parameters. (See \cite{Shang})
\end{itemize}

\end{document}